\documentclass{amsart}
\usepackage{amsmath}
\usepackage{amssymb}
\usepackage{amsthm,amssymb,amsmath,enumerate,graphicx}
\usepackage{color,soul}
\usepackage{subfigure}
\usepackage{caption}
\usepackage{epstopdf}

\newtheorem{theorem}{Theorem}[section]

\newtheorem{conjecture}[theorem]{Conjecture}

\newtheorem{proposition}[theorem]{Proposition}

\theoremstyle{remark}
\newtheorem{notation}[theorem]{Notation}
\newtheorem{remark}[theorem]{Remark}

\newtheorem{claim}{Claim}

\newcommand{\F}{\mathcal F}

\begin{document}

\makeatletter

\makeatother
\author{Ron Aharoni}
\address{Department of Mathematics\\ Technion, Haifa, Israel}
\email[Ron Aharoni]{raharoni@gmail.com}
\thanks{\noindent The research of the first author was
supported by BSF grant no. $2006099$, by GIF grant no. I
$-879-124.6/2005$, by the Technion's research promotion fund, and by
the Discont Bank chair.}

\author{Dani Kotlar}
\address{Computer Science Department, Tel-Hai college, Upper Galilee, Israel}
\email[Dani Kotlar]{dannykotlar@gmail.com}

\author{Ran Ziv}
\address{Computer Science Department, Tel-Hai college, Upper Galilee, Israel}
\email[Ran Ziv]{ranzivziv@gmail.com}

\title{Representation of large matchings in bipartite graphs}
\maketitle
\begin{abstract}
Let $f(n)$ be the smallest number such that every collection of $n$
matchings, each of size at least $f(n)$, in a bipartite graph, has a
full rainbow matching. Generalizing famous conjectures of Ryser,
Brualdi  and Stein,  Aharoni and Berger \cite{AhBer} conjectured
that $f(n)=n+1$ for every $n>1$.  Clemens and Ehrenm{\"u}ller
\cite{Cl-Eh} proved that $f(n) \le \frac{3}{2}n +o(n)$. We show that
the $o(n)$ term can be reduced to a constant, namely $f(n) \le
\lceil \frac{3}{2}n \rceil+1$.
\end{abstract}

\section{Introduction}

Given sets $F_1,F_2,\ldots,F_n$ of edges in a graph, a
\emph{(partial) rainbow matching} is a choice of disjoint edges from
some of the $F_i$s.  In other words, it is a partial choice function
whose range is a matching. If the rainbow matching represents all
$F_i$s then we say that it is \emph{full}. For a comprehensive
survey on rainbow matchings and the related subject of transversals in Latin squares see
\cite{Wan11}.

As in the abstract, we assume the graph is bipartite and define
$f(n)$ to be the  least number such that if $|F_i| \ge f(n)$ for all
$i=1,\ldots,n$, then there exists a full rainbow matching.
 A greedy choice of representatives shows that if $|F_i|\ge 2n-1$ for all $i=1,\ldots,n$
 then  there is a rainbow matching. Thus, $f(n)\le 2n-1$.
 On the other hand, for every $n>1$ there exits a  family $F_1, \ldots, F_n$ of matchings of size $n$ with no full rainbow matching:
 for an arbitrary $1 \le k \le n$  let $F_1, \ldots, F_{k}$ be all equal to the perfect matching in the cycle
$C_{2n}$ consisting of the odd edges,  and let $F_{k+1}, \ldots,
F_{n}$ be all equal to the perfect matching in $C_{2n}$ consisting
of the even edges. This shows
that $f(n)\ge n+1$ for all $n>1$ (in fact, this example can be modified to produce $2n-2$ matchings of size $n$ with no rainbow matching of size $n$). In \cite{AhBer} it was conjectured
that this bound is sharp:

\begin{conjecture}\label{main} \cite{AhBer} $f(n)=n+1$ for all $n>1$. \end{conjecture}

If true, this would easily imply:

 \begin{conjecture}\label{conj3}
A family of $n$ matchings in a bipartite graph, each of size $n$,  has a rainbow matching of size $n-1$.
\end{conjecture}

This strengthens a famous conjecture of Ryser-Brualdi-Stein.

\begin{conjecture}\label{conj1}\cite{BruRys,Ryser67,Stein75}
A partition of the edges of the complete bipartite graph $K_{n,n}$ into $n$ matchings, each of size $n$, has a rainbow matching of size $n-1$.
\end{conjecture}

Another strengthening of the last conjecture is due to Stein:

\begin{conjecture}\cite{Stein75} \label{conj2}
A partition of the edges of the complete bipartite graph $K_{n,n}$ into n subsets, each of size $n$, has a rainbow matching of size $n-1$.
\end{conjecture}

In our terminology, the weaker condition that Stein demands on sets $F_i$ is not that they are matchings, but that each has degree at  most $1$ in one side of the graph, and that  jointly their degree at each vertex in the other side is at most $n$. Possibly the `right' requirement is even more general: that the  degree at each vertex is at most $n$, and that each $F_i$ is a set, and not a multiset, namely it does not contain repeating edges.

Successive improvements on the trivial bound $f(n)\le 2n-1$
were $f(n)\le\lfloor\frac{7}{4}n\rfloor$
\cite{ACH}, $f(n)\le \lfloor\frac{5}{3}n\rfloor$ \cite{KZ-EJC}  and
$f(n)\le\lfloor\frac{3}{2}n\rfloor + o(n)$ \cite{Cl-Eh}. The latter
was extended in \cite{cep} to general graphs, and to the more
general case in which the sets $F_i$ are not assumed to be matchings, but disjoint unions
of cliques,  each containing $3n+o(n)$ vertices.
Pokrovskiy \cite{Pok} showed that if we add the requirement
that the $n$ matchings are edge disjoint, then $|F_i|\ge n +o(n)$
suffices. In this note we prove:

\begin{theorem}\label{thm:3:2}
$f(n)\le\lceil\frac{3}{2}n\rceil +1$. \end{theorem}

% ----------------------------------------------------------------
\section{Proof of Theorem~\ref{thm:3:2}}\label{three-halves}

 The following was shown in \cite{KZ-EJC}:

\begin{proposition}\label{prop1}
A family $\F =\{F_1,\ldots,F_n\}$  of $n$ matchings in a bipartite graph, each of size  at least $\lfloor\frac{3}{2}n\rfloor$,  has a rainbow matching of size $n-1$.
\end{proposition}

\begin{proof}[Proof of Theorem~\ref{thm:3:2}]
Let $G$ be the given bipartite graph and let $U,W\subset V(G)$ be the two sides of $G$. Let $\F =\{F_1,\ldots,F_n\}$ be a family of matchings in $G$, each of size at least $\lceil\frac{3}{2}n\rceil +1$, and
let $R$ be a rainbow matching of maximal size. By Proposition \ref{prop1}, $|R| \ge n-1$. We assume, for contradiction, that  $|R|=n-1$ and without loss of generality we may assume that $R\cap F_n=\emptyset$. For each $i=1,\ldots,n-1$ let $F_i\cap R=\{r_i\}$ and let $r_i=\{u_i,w_i\}$, where $u_i\in U$ and $w_i\in W$. Let $X\subset U$ and $Y\subset W$ be the sets of vertices of $G$ not matched by $R$. We shall use the following notation:

\begin{notation}
For any two sets of vertices $A\subseteq U$ and $B\subseteq W$ we denote by $E(A,B)$ or $E(B,A)$ the set of edges in $E(G)$ with one endpoint in $A$ and the other endpoint in $B$.
\end{notation}

Let $F_n^Y$ be the subset of $F_n$ consisting of edges matching vertices in $Y$. Since $R$ has maximal size, $F_n^Y\subset E(Y,U\setminus X)$.
Let $U'$ be the set of vertices in $U\setminus X$ that are endpoints of the edges in $F_n^Y$. Let $R'$ be the subset of $R$ that matches the vertices in $U'$, and let $W'$ be the set of vertices in $W$ that are endpoints of edges in $R'$ (the set $U'$ is matched by $R'$ to $W'$). The main idea of the proof is to replace some edges in $R'$ by edges in $E(X,W\setminus Y)$, thus freeing vertices in $U'$. This will allow us to add an edge from $F_n^Y$ to the rainbow matching.

Let $\ell=|F_n^Y|$. Since $|W\setminus Y| = n-1$ and $|F_n|\ge \lceil 3n/2\rceil+1$ we have $\ell\ge \lceil n/2\rceil+2$. By possibly ignoring some edges of $F_n$ we shall assume that
\begin{equation}\label{eq1:0}
    \ell = \lceil n/2\rceil+2.
\end{equation}
So,
\begin{equation}\label{eq1:1}
    |U'|=|W'|=|R'| = \lceil n/2\rceil+2.
\end{equation}
Define,
\begin{equation*}
    \F'=\{F_i\in\F | F_i\cap R'\ne \emptyset\}.
\end{equation*}

That is, $\F'$ consists of the matchings that are represented in the partial rainbow matching $R'$.

\begin{notation}
For each $F_i\in\F'$ let $e_i$ be the edge of $F_n^Y$ such that $e_i\cap r_i\ne \emptyset$. Let $y_i$ be the endpoint of $e_i$ in $Y$ (Figure~\ref{fig1}).

\end{notation}

\begin{figure}[h!]
\begin{center}
\includegraphics[scale=0.2]{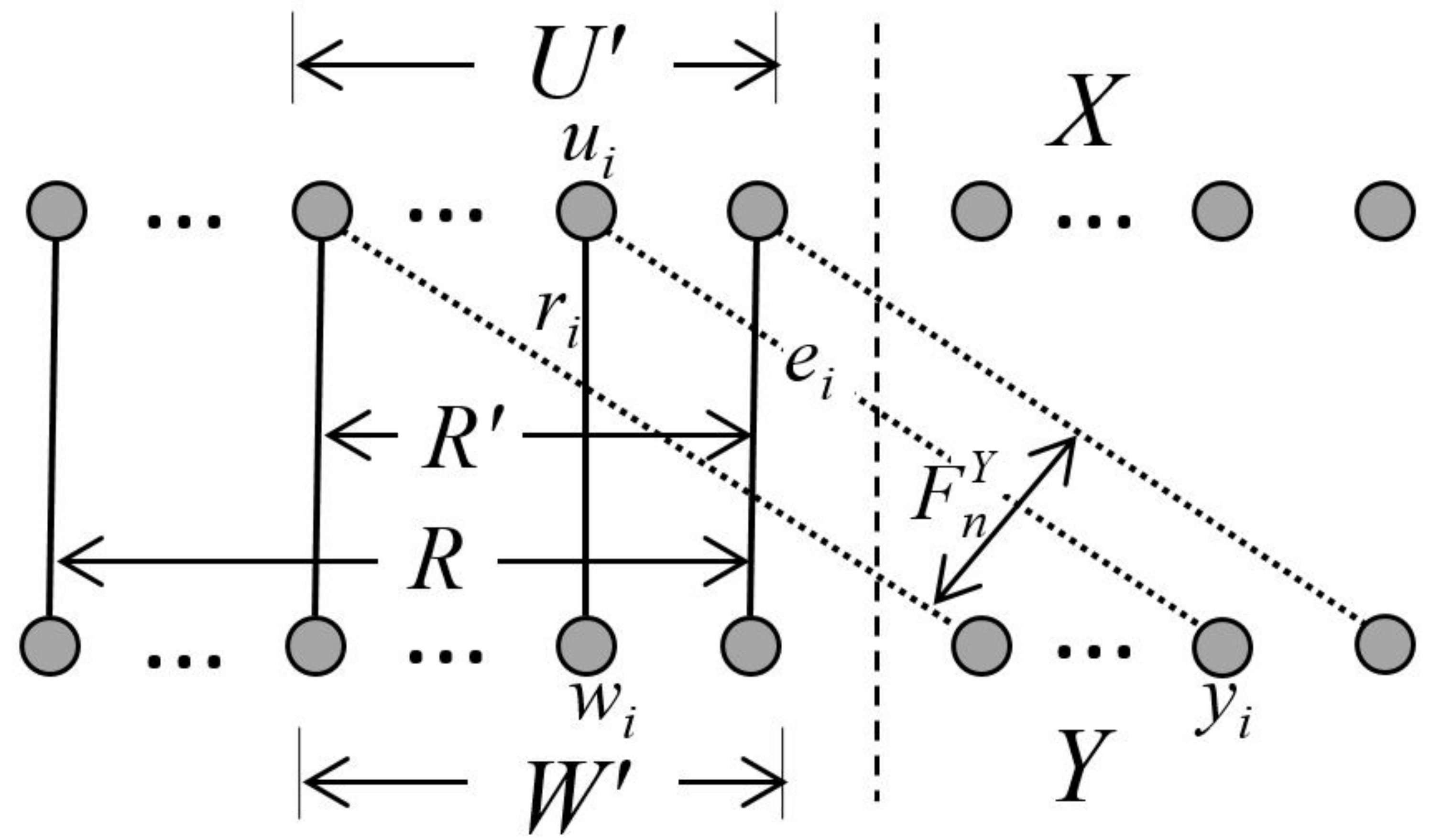}
\end{center}
\caption{}
\label{fig1}
\end{figure}

\begin{claim}\label{claim:1}
For each $F_i\in\F'$, we have $|F_i\cap E(X, W\setminus Y)|\ge\lceil n/2\rceil+1$ and $|F_i\cap E(Y, U\setminus X)|\ge \lceil n/2\rceil+1$.
\end{claim}

\begin{proof}%[Proof of Claim~\ref{claim:1}]
We show that each $F_i\in\F'$ has at most one edge between $X$ and $Y$.
Suppose $F_i$ has two edges $e$ and $f$ between $X$ and $Y$. The edge $e_i$ is disjoint from one of them, say $e$. Thus, $\left(R\setminus \{r_i\}\right)\cup\{e_i,e\}$ is a rainbow matching of size $n$, contradicting the maximality of $R$ (Figure~\ref{fig2a}).
\end{proof}

\begin{figure}[h!]
\begin{center}
\includegraphics[scale=0.2]{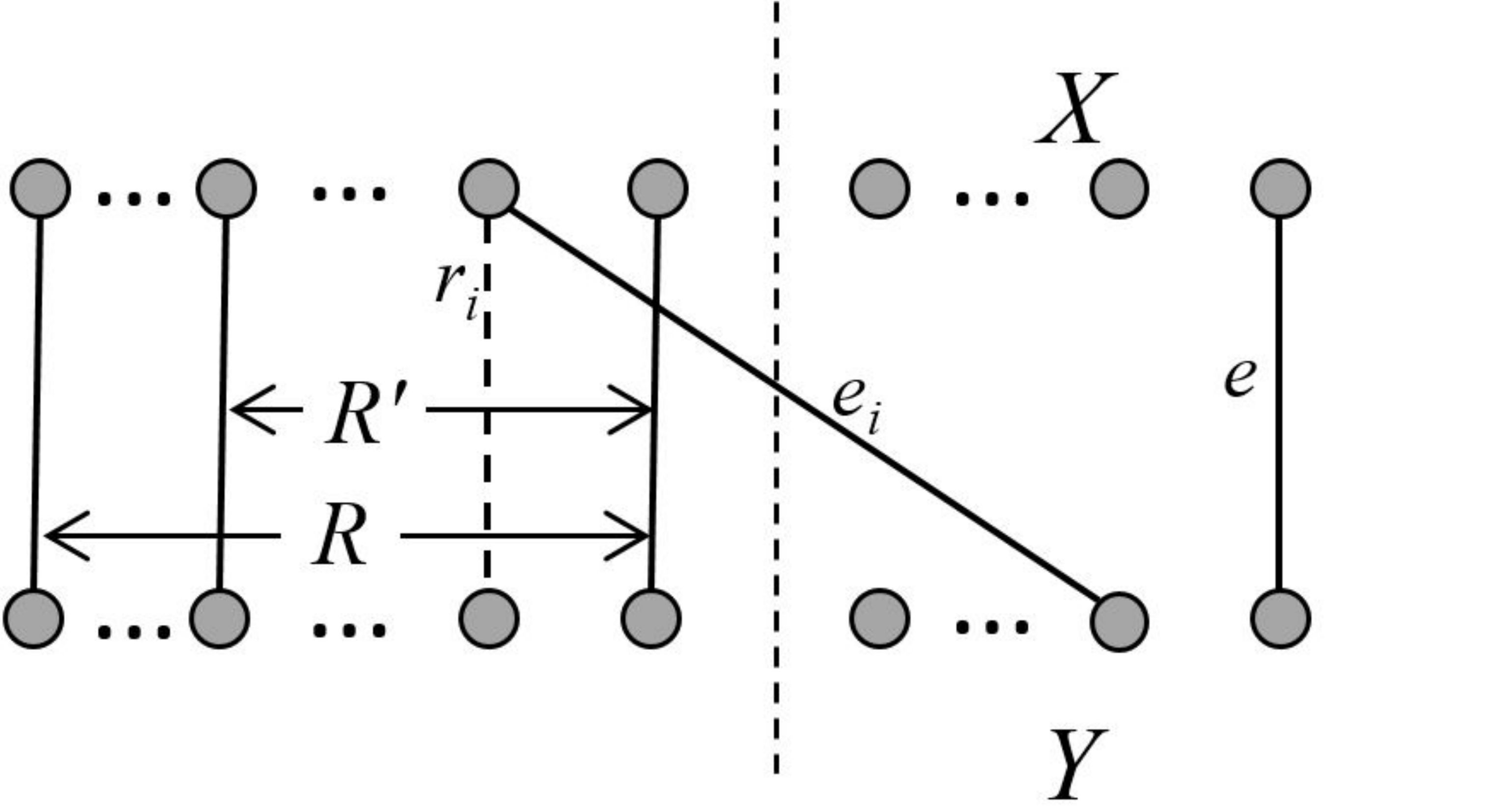}
\end{center}
\caption{}
\label{fig2a}
\end{figure}

\begin{remark} In all the figures, dashed lines represent edges that are candidates to be removed from the rainbow matching, and solid and dotted lines represent edges that are candidates for being added in.\end{remark}

% \begin{figure}[h!]
%  \centering
%  \subfigure[]{\label{fig2a}\includegraphics[scale=0.2]{fig2a.pdf}}
%  \subfigure[]{\label{fig2b}\includegraphics[scale=0.2]{fig2b.pdf}}
%  \caption{}
%  \label{fig2}
% \end{figure}

\begin{notation}\label{not:1}
For each $F_i\in\F'$ we denote $F_i^Y=F_i\cap E(Y\setminus\{y_i\},U\setminus X)$. Let
$U^*$ be the union of $U'$ and the set of vertices in $U\setminus X$ that are endpoints of edges in $\bigcup \{F_i^Y \mid F_i \in \F'\}$. Let $R^*$ be the subset of $R$ that matches the elements in $U^*$ and let $W^*$ be the set of vertices in $W$ that are matched by $R^*$. We define

\begin{equation*}
    \F^*=\{F_j\in\F | F_j\cap R^*\ne \emptyset\}.
\end{equation*}

(Note that $U'\subset U^*\subset U\setminus X$, $W'\subset W^*\subset W\setminus Y$, $R'\subset R^*\subset R$ and $\F'\subset \F^*\subset \F$.) Let $\F''=\F^*\setminus\F'$ and let $d=|\F''|$ (it is possible that $d=0$).
\end{notation}

\begin{claim}\label{claim:2}
For each $F_j\in\F''$, $|F_j\cap E(X, W\setminus Y)|\ge\lceil n/2\rceil$ and $|F_j\cap E(Y, U\setminus X)|\ge \lceil n/2\rceil$.
\end{claim}

\begin{proof}
Let $F_j\in\F''$. We show that $F_j$ has at most two edges between $X$ and $Y$. By the definition of $\F^*$, there exists $F_i\in\F'$ and an edge $f\in F_i$ such that $f\cap r_j=\{u_j\}\subset U\setminus X$ and the other endpoint $y$ of $f$ is in $Y\setminus\{y_i\}$. Now suppose $F_j$ has three edges between $X$ and $Y$. Then one of them, say $e$, has an endpoint in $Y\setminus\{y_i,y\}$. Now, $R\setminus\{r_i,r_j\}\cup\{f,e_i,e\}$ is a rainbow matching, contradicting the maximality of $R$ (Figure~\ref{fig2b}).
\end{proof}

\begin{figure}[h!]
\begin{center}
\includegraphics[scale=0.2]{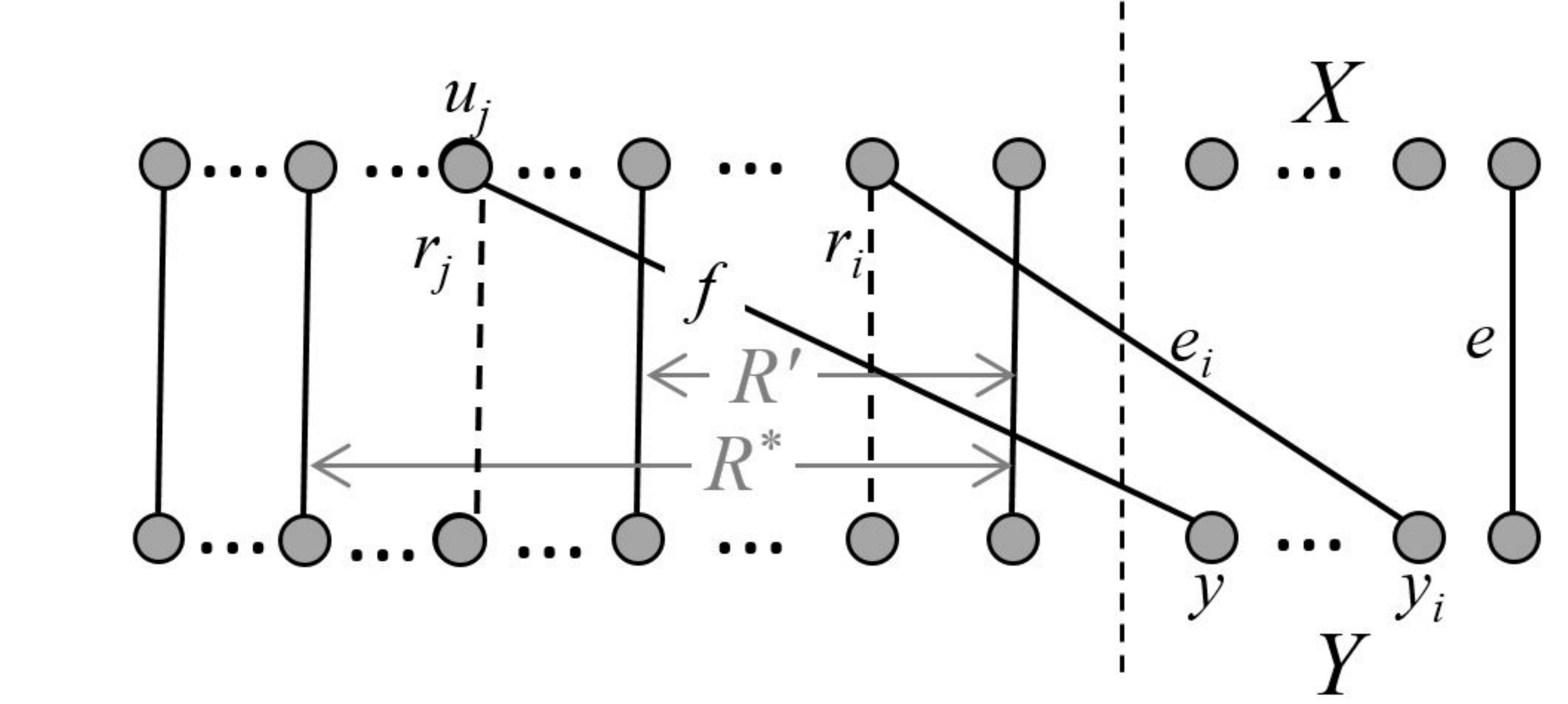}
\end{center}
\caption{}
\label{fig2b}
\end{figure}

\begin{claim}\label{claim:3}
For each $F_i\in\F^*$, $|F_i\cap E(X,W^*)|\ge d+3$.
\end{claim}

\begin{proof}%[Proof of Claim~\ref{claim:3}]
Since $|R^*|=|R'|+d$, it follows by (\ref{eq1:1}), that
$|R\setminus R^*| = n-1- (\lceil n/2\rceil+2+d)=\lfloor
n/2\rfloor-d-3$.  Let $F_i\in\F^*=\F'\cup\F''$ (disjoint union). Since $F_i\cap E(X,W^*) = F_i\cap E(X,W\setminus Y) \cap R^*$, we have by Claims~\ref{claim:1} and \ref{claim:2}, that $F_i\cap E(X,W^*) \ge \lceil n/2\rceil- (\lfloor n/2\rfloor-d-3)\ge d+3$.
\end{proof}

We shall inductively choose edges $f_1,f_2,\ldots,f_i\in E(X,W^*)$ from distinct $F_j$s and $r_1,r_2,\ldots,r_i, r_{i+1}\in R^*$ from distinct $F_j$s, as follows.
To start the process, we assume, without loss of generality, that $F_1\in\F'$. By Claim~\ref{claim:3}, there exists $f_1\in F_1\cap E(X,W^*)$. Let $w_2$ be the endpoint of $f_1$ in $W^*$, and without loss of generality we may assume that $w_2\in r_2$, where $r_2 \in R^* \cap F_2$. Again, by Claim~\ref{claim:3}, there exists $f_2\in F_2\cap E(X\setminus\{x_1\},W^*)$.
We continue in this manner, choosing at each step an edge $f_i\in E(X,W^*)$, disjoint from all $f_j, ~j <i$, and belonging to the same matching as $r_i$, and the edge $r_{i+1} \in R^*$, such that $f_i\cap r_{i+1}\cap W^*\ne \emptyset$. The process ends when we have obtain a set of disjoint edges $F=\{f_1,f_2,\ldots,f_m\}\subseteq E(X,W^*)$ and a set of distinct edges $P=\{r_1,r_2,\ldots,r_m, r_{m+1}\}\subseteq R^*$ such that $f_i\cap r_{i+1}\cap W^*\ne \emptyset$ for $i=1,\ldots,m$, and for each $i$, $f_i$ and $r_i$ belong to the same matching (without loss of generality we assume that $f_i,r_i\in F_i$ for $i=1,\ldots,m$, and $r_{m+1}\in F_{m+1}$), so that one of two options holds:
\begin{enumerate}
  \item [(1)] $m<d+3$ and the matching $F_{m+1}$ has an edge $f_{m+1}\in E(X\setminus(f_1\cup f_2\cup\ldots\cup f_m),W^*)$ such $f_{m+1}\cap r_t\cap W^*\ne\emptyset$ for some $t\in\{1,\ldots,m\}$, or
  \item [(2)] $m= d+3$.
\end{enumerate}
(Note that by Claim~\ref{claim:3} one of these two options must hold.)

  In Case (1) the partial rainbow matching $R$ can be augmented as follows:
  If $F_i\in\F'$ for some $i\in\{t,\dots,m+1\}$, then $\left(R\setminus\{r_t,\ldots,r_{m+1}\}\right)\cup \{f_t,\dots,f_{m+1},e_i\}$ is a full rainbow matching (Figure~\ref{fig3a}).
  If $F_i\in\F''$ for all $i\in\{t,\dots,i+1\}$, then, by the definition of $\F^*$, there exists $F_j\in\F'$ and an edge $e\in F_j^Y$ so that $e\cap r_t\in U^*$. In this case $\left(R\setminus\{r_t,\ldots,r_{m+1},r_j\}\right)\cup \{f_t,\dots,f_{m+1},e,e_j\}$ is a full rainbow matching (Figure~\ref{fig3b}). (Note that $e$ and $e_j$ are disjoint by the definition of $F_j^Y$.)

 \begin{figure}[h!]
  \centering
  \subfigure[]{\label{fig3a}\includegraphics[scale=0.2]{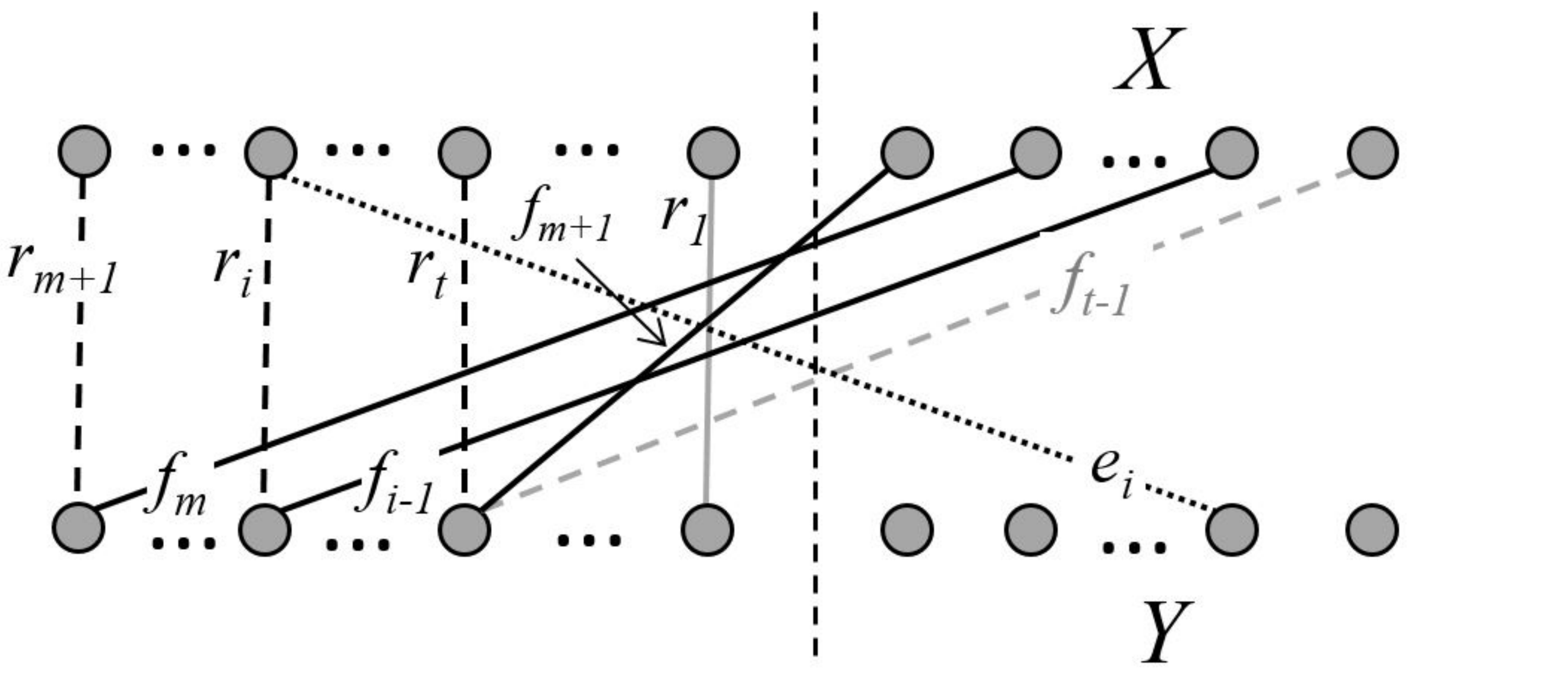}}
  \subfigure[]{\label{fig3b}\includegraphics[scale=0.2]{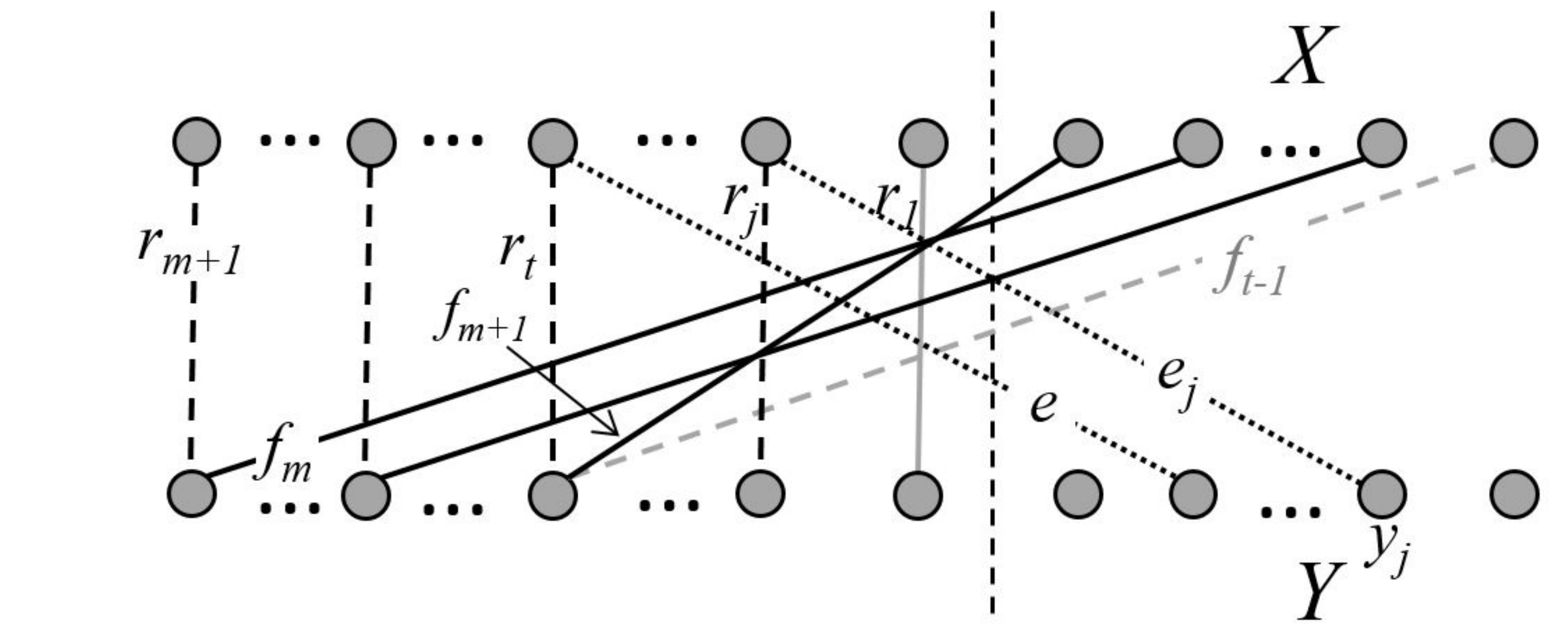}}
  \caption{}
  \label{fig4}
\end{figure}

  In Case (2) let $Q=\left(R\setminus P\right)\cup F$. Then, $Q$ is a partial rainbow matching of size $n-2$, since it excludes the matchings $F_{m+1}$ and $F_n$. We shall augment $Q$ with edges two edges in $E(Y,U^*)$, from $F_{m+1}$ and $F_n$ respectively.

\begin{claim}\label{claim:4}
If $F_i\in\F'$, then the size of the set $\{e\in F_i^Y: e\cap (\cup_{j=1}^m r_j)\ne \emptyset\}$ is at least 2.
\end{claim}

\begin{proof}
Let $U^i$ be the set of endpoints in $U\setminus X$ of the edges in $F_i^Y$. Note that $|U^i|\ge \ell-1$ (Claim~\ref{claim:1} and \eqref{eq1:0}), $U^i\subset U^*$ (since $F_i\in \F'$), and $|U^*|=\ell+d$. Recall that for each edge $r_j\in R\cap F_j$ its endpoint in $U\setminus X$ was denoted $u_j$. Since $|U^*\setminus\{u_1,\ldots,u_{m}\}|=\ell+d-m=\ell+d-(d+3)=\ell-3$, the claim follows.
\end{proof}

 There are two sub-cases to consider: (2a) $F_{m+1}\in \F'$, and (2b) $F_{m+1}\in \F''$.

   (2a) Assume $F_{m+1}\in \F'$. By Claim~\ref{claim:4}, there exists and edge $e \in F_{m+1}$ connecting a vertex in $Y\setminus\{y_{m+1}\}$ with some $u_t$, which is the endpoint in $U$ of some $r_t\in P\setminus\{r_{m+1}\}$. Since $m= d+3$ and $|P|=m+1= d+4$, at least four of the edges in $P$ are in $R'$ (actually, three are enough in this case). For at least one of these four edges, say $r_i$, its corresponding $e_i$ (the edge of $F_n^Y$ meeting $r_i$ in $U$) avoids both endpoints of $e$. Then, $Q\cup\{e,e_i\}$ is a rainbow matching of size $n$ (Figure~\ref{fig4a}).

 (2b) Assume $F_{m+1}\in \F''$ By Claim~\ref{claim:2}, $|F_{m+1}^Y|\ge\lceil n/2\rceil$. Since by \eqref{eq1:1} we have $|R\setminus R'| = n-1-(\lceil n/2\rceil+2)=\lfloor n/2\rfloor-3$, there is an edge $e\in F_{m+1}^Y$ sharing an endpoint with an edge $r_s\in R'$. Assume first that $s\in \{1,\ldots,m\}$. As in the previous paragraph, there exists $e_i$ disjoint from $r_s$ and $e$, so that $Q\cup\{e,e_i\}$ is a rainbow matching of size $n$. Now assume that $s\not\in \{1,\ldots,m\}$ and let again $e$ be the edge of $F_{m+1}$ sharing an endpoint with $r_s$. Since $r_s\in R'$, there exists, by Claim~\ref{claim:4}, an edge $e'\in F_s^Y$, disjoint from $e$, sharing an endpoint with some $r_t$ with $t\in\{1,\ldots,m\}$. Since $|P\cap R'|\ge 4$, there exists an edge $e_i\in F_n^Y$, avoiding both endpoints of $e'$ and the endpoint of $e$ in $Y$, such that $u_i\in\{u_1,\ldots,u_{m+1}\}$. Then, $Q\setminus\{r_t\}\cup\{e,e',e_i\}$ is a rainbow matching of size $n$ (Figure~\ref{fig4b}). This completes the proof.

 \begin{figure}[h!]
  \centering
  \subfigure[]{\label{fig4a}\includegraphics[scale=0.2]{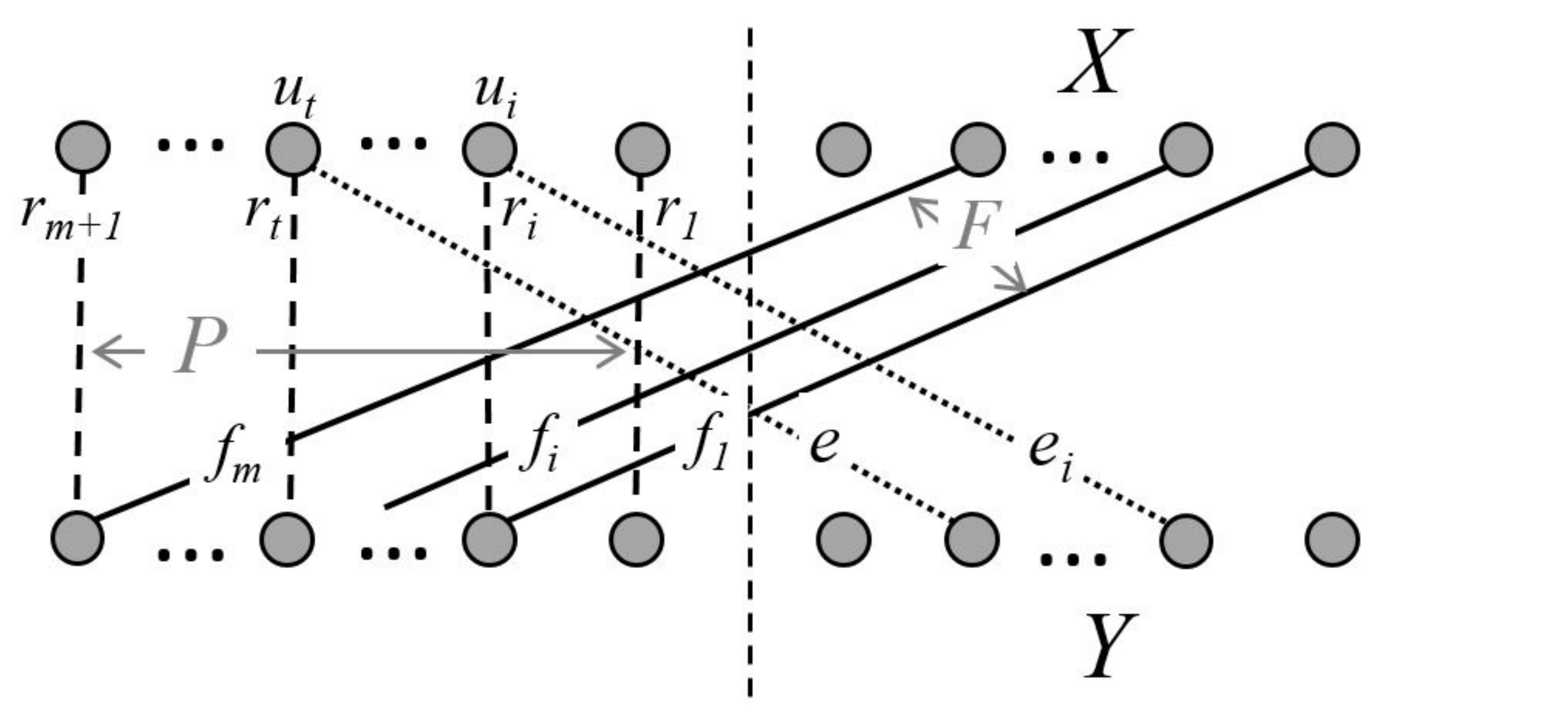}}
  \subfigure[]{\label{fig4b}\includegraphics[scale=0.2]{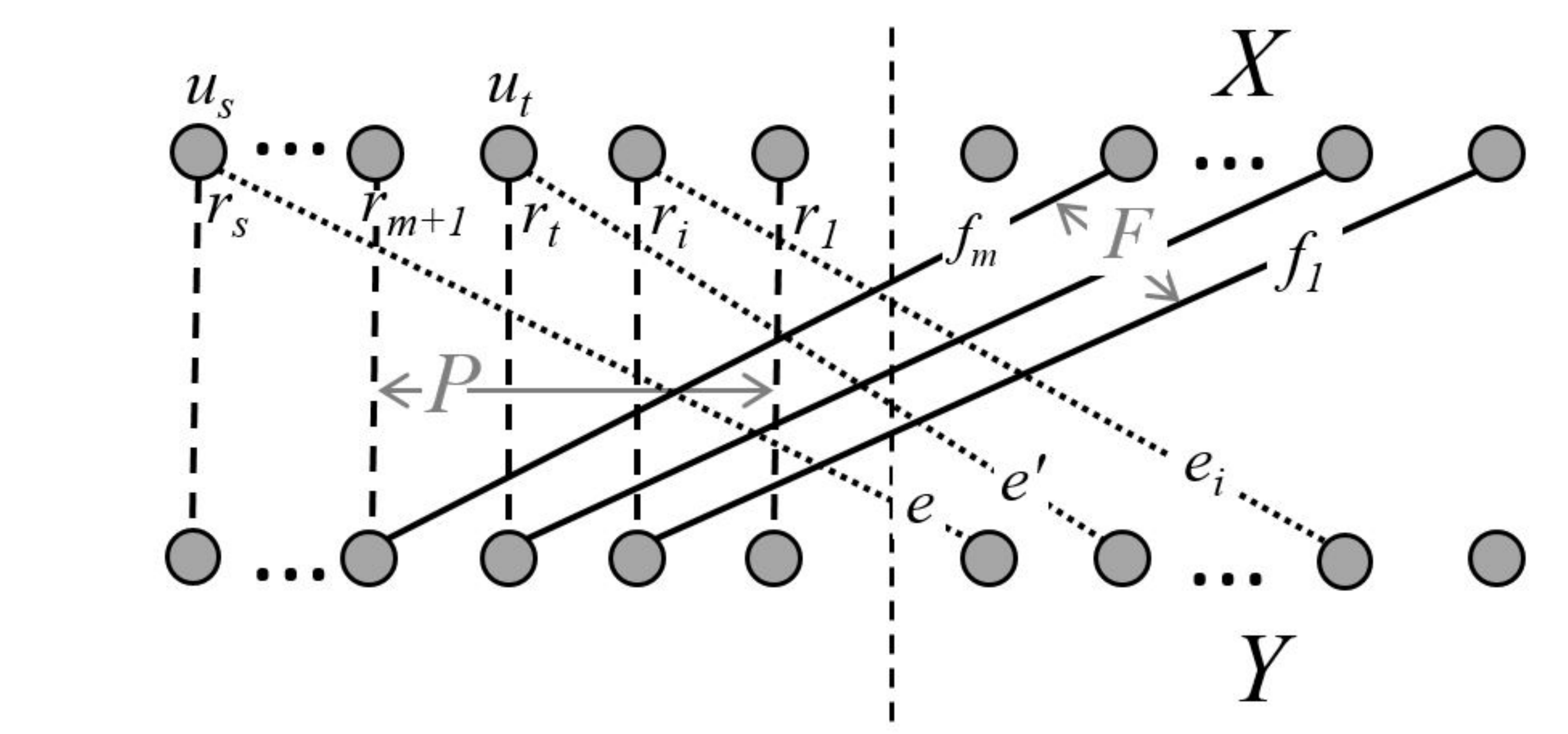}}
  \caption{}
  \label{fig4}
\end{figure}

\end{proof}

% --------------------------------------------------------------------------------------------------------------------
%\bibliographystyle{amsplain}
%\bibliography{rainbow_bib}

\providecommand{\bysame}{\leavevmode\hbox to3em{\hrulefill}\thinspace}
\providecommand{\MR}{\relax\ifhmode\unskip\space\fi MR }
\providecommand{\MRhref}[2]{%
  \href{http://www.ams.org/mathscinet-getitem?mr=#1}{#2}
}
\providecommand{\href}[2]{#2}

\end{document}